\documentclass[10pt]{article}
\usepackage{amsmath, amssymb ,amsthm, amsfonts, amsgen}
\usepackage{graphicx}
\usepackage{rotate}
\usepackage[dvips]{color}

\numberwithin{equation}{section} 

\newcommand{\A}{{\mathcal{A}}}

\makeatletter
\def\rightharpoonupfill@{\arrowfill@\relbar\relbar\rightharpoonup}
\newcommand{\xrightharpoonup}[2][]{\ext@arrow
0359\rightharpoonupfill@{#1}{#2}} \makeatother

\def\dist{\text{dist}}

\let\e=\varepsilon

\def\e{{\varepsilon}}
\def\O{{\Omega}}
\def\o{{\omega}}

\def\A{{\cal A}}

\newtheorem{Theorem}{Theorem}[section]
\newtheorem{Lemma}[Theorem]{Lemma}
\newtheorem{Proposition}[Theorem]{Proposition}

\newtheorem{Definition}[Theorem]{Definition}

\newcommand{{\supess}}{\mathop{\rm ess\: sup }}
\newcommand{{\rr}}{{\mathbb R}}

\newenvironment{@abssec}[1]{%
     \if@twocolumn
       \section*{#1}%
     \else
       \vspace{.05in}\footnotesize
       \parindent .2in
         {\upshape\bfseries #1. }\ignorespaces
     \fi}
     {\if@twocolumn\else\par\vspace{.1in}\fi}
\newcommand\keywordsname{Key words}
\newcommand\AMSname{AMS subject classifications}

\begin{document}

\title{A remark about dimension reduction for supremal functionals: the case with convex domains.}
\author{\textsc{ Elvira Zappale}\thanks{Universita'
degli Studi di Salerno, Via Ponte Don Melillo, 84084 Fisciano (SA) Italy.
E-mail:ezappale@unisa.it}}
\maketitle

\begin{abstract}
\noindent An application of dimensional reduction results for gradient constrained problems is provided for 3D-2D dimension reduction for supremal functionals, in the case when the domain is convex.
 
\medskip

\noindent\textbf{Keywords}:  $\Gamma$-convergence, supremal functionals, gradient constrained problems, dimension reduction.

\noindent\textbf{MSC2010 classification}:  49J45, 74K99, 74Q99.

\end{abstract}

\section{Introduction}

Dimension reduction deals with the study of the asymptotic
behavior of the solutions of a partial differential equation (or a minimization
problem) stated on a domain where one of the dimensions is much smaller than the others. 
This kind of problems has been widely treated  by means of $\Gamma$-convergence analysis. Indeed $\Gamma$-convergence, (see  subsection \ref{Gammaconv} for definitions and \cite{D} for details),
is well suited for dealing with the asymptotic behavior of variational
problems depending on a parameter, since it provides good informations on
the asymptotics of minimizers and of the minimal values.

The aim of this paper is to provide a new proof for a $\Gamma$-convergence result, with respect to uniform convergence, for a family of functionals $H_\e: W^{1,\infty}(\Omega)\to \mathbb R$, of the type
\begin{equation}\label{He}
\displaystyle{H_\e(u):=\supess_{\Omega} W\left(\nabla_\alpha u, \frac{1}{\e}\nabla_3 u\right)},
\end{equation}
where $\Omega \subset \mathbb R^3$ is a bounded open set of cylindrical shape, $W:\mathbb R^3 \to R$ is a  continuous function which satisfies suitable growth assumptions, $\nabla_\alpha u$ and $\nabla_3 u$ stand for the partial derivatives of $u$ with respect to $x_\alpha\equiv(x_1,x_2)$ and $x_3$ respectively.

Functionals $H_\e$ are called supremal functionals  (see subsection \ref{supremal}). They were introduced in \cite{aron4} in order to study the problem of finding the best Lipschitz extension and they are deeply connected with the study of $-\Delta_\infty$ (cf. \cite{ACJ}). The properties of such
functionals proved to be useful in order to give a mathematical model for many physical problems as, for example, the problem of modelling the dielectric breakdown for a composite conductor or for the description of plasticity in polycrystals, cf. \cite{GNP} and \cite{BN}.

It is well known that the asymptotic behaviour of the functionals $\{H_\e\}$ is related to the following minimum problems
$$
\inf \left\{\supess_{\Omega(\varepsilon)} W(\nabla u): \; u \in W^{1,\infty}(\Omega(\varepsilon))\right\},
$$
where $\omega \subset \mathbb R^2$ is a bounded open set, satisfying suitable assumptions that will be specified later, $\Omega(\varepsilon):= \omega \times (-\varepsilon, \varepsilon)$, $W$ is as in \eqref{He} and $\nabla u$ stands for the vector of all the partial derivatives of $u$.

In fact, rescaling the domain $\Omega(\e)$ one is lead to consider 
\begin{equation}\nonumber
\inf\left\{ \supess_\Omega W \left(\nabla_\alpha u, \frac{1}{\varepsilon}\nabla_3 u\right): u \in W^{1,\infty}(\Omega) \right\},
\end{equation}
with $\Omega:= \Omega(1)=\omega \times (-1,1)$.

The main results are contained in the following theorems.
\begin{Theorem}\label{myproof}
Let $\omega \subset \mathbb R^2$ be a convex, bounded open set and let $\Omega:= \omega \times (-1,1)$. Let $W:\mathbb R^3 \to [0,+\infty)$ be a continuous function, and assume that
\begin{equation}\label{Ccoercivity}
\displaystyle{W(\xi )\geq C \|\xi\| \;\;\; \hbox{ for every } \xi \in \mathbb R^3.}
\end{equation}
For every $\e>0$ define
$F_\e: C(\overline \Omega) \to [0, +\infty]$ by
\begin{equation}\nonumber
F_\e(u):=\left\{
\begin{array}{ll}
\displaystyle{
\supess_{x \in \Omega} 
W\left(\nabla_\alpha u(x), \frac{1}{\e} \nabla_3 u(x) \right)} &\hbox{ if }u \in W^{1,\infty}(\Omega),\\
\\
+\infty &\hbox{ otherwise.}
\end{array}
\right.
\end{equation}
Then the family $\{F_\e\}_\e$ $\Gamma(L^\infty)$-converges to the functional $F_0: C(\overline \Omega) \to [0,+\infty]$ given by
\begin{equation}\label{Gammaesssup}
F_0(u):=\left\{\begin{array}{ll}
\displaystyle{\supess_{x_\alpha \in \omega} (W_0)^{\rm lc}(\nabla_\alpha u(x_\alpha))} & \hbox{ if  }u \in W^{1,\infty}(\omega),\\
\\
+\infty &\hbox{otherwise,}
\end{array}
\right.
\end{equation}
where, for every $z \in \mathbb R^2$,
$$
\displaystyle{W_0(z)=\inf_{\zeta \in \mathbb R}W(z, \zeta),}
$$
and $(W_0)^{\rm lc}$ is the level convex envelope of $W_0$,
defined in \eqref{lcenv}.

\end{Theorem}

The argument we show in the proof of Theorem \ref{myproof}  relies on some results about dimensional reduction for gradient constrained integral functionals, developed in an unpublished manuscript by R. De Arcangelis. These results are contained in \cite{DAF} and they deal with $\Gamma$-convergence and integral representation in the framework of $BV$-spaces, while the version we present in the theorem below is suited for the applications to the supremal functionals in the continuous setting.

\begin{Theorem}
\label{thmDAFL-Z}
Let $\omega$ be a convex bounded open set in $\mathbb R^2$ and let $\Omega:= \omega \times (-1,1)$. Let $f:\mathbb R^3 \to [0, +\infty]$ be a Borel function, 
Let $F_{f_\e}(\cdot ,\Omega): W^{1,\infty}(\Omega)\to [0, +\infty[$ be the functional defined as
\begin{equation}\label{Ffe}
F_{f_\varepsilon}(u,\omega):=
\displaystyle{\int_{\Omega}f\left(\nabla_\alpha u,\frac{1}{\varepsilon}\nabla_3 u\right)dx,}
\end{equation}
It results that $F_{f_\varepsilon}$ $\Gamma (L^\infty)$- converges, as $\varepsilon \to 0$, to 
$$\overline{F}(u, \omega):=2\int_\omega (f_0)^{\ast \ast}(\nabla_\alpha u)dx_\alpha \hbox{ for every }u \in W^{1,\infty}(\omega),$$ where $f_0$ is defined in \eqref{1.18R} and $(f_0)^{\ast \ast} $ represents its lower semicontinuous and convex envelope as in \eqref{astast}. 
\end{Theorem}

We stress the fact that  $F_{f_\varepsilon}(u,\Omega)$ is finite if and only if $\nabla u \in {\rm dom}f_\varepsilon$ a.e. in $\Omega$, i.e. 
\begin{equation}\label{remunb}
\left(\nabla_\alpha u, \frac{1}{\varepsilon}\nabla_3 u\right) \in {\rm dom }f \hbox{ a.e. in } \Omega.
\end{equation}

Finally we underline that the target of our contribution consists in showing that also in the case of dimension reduction  it is exhibited the deep relation existing between unbounded functionals and supremal functionals (see section \ref{unbounded}), as emphasized in \cite{BJW99} and exploited in \cite{BGP}.

The paper is organized as follows: section \ref{preliminaries} is devoted to preliminary results of Convex Analysis, measure thoery, $\Gamma$- convergence and supremal functionals. In section \ref{unbounded} we present the results, first we treat  the integral representation of unbounded functionals, then in subsection \ref{thm1.1} we prove theorem \ref{myproof}.

\section{Preliminaries} \label{preliminaries}

We start by fixing notations and recalling results from Convex Analysis.

\noindent
$N \in \mathbb N$ is fixed.
${\cal A}_0(\mathbb R^N)$ denotes the sets of the bounded open subsets of ${\mathbb R}^N$, and
${\cal L}^N$
is the Lebesgue measure on $\mathbb R^N$.

For a given $S\subseteq {\mathbb R}^N$ we denote by ${\it aff} (S)$ 
the affine hull of ${\it aff} S$, defined as the 
intersection of all the affine
sets containing $S$. It is clear that
${\it aff} (S)$ is the smallest affine set containing $S$.

For every $S\subseteq {\mathbb R}^N$ we denote by ${\rm co}(S)$ the 
convex hull of $S$, i.e. the intersection
of all the convex subsets of
$\mathbb R^N$ containing $S$. It is clear that ${\rm co}(S)$ 
is the smallest convex set containing $S$.

If $C\subseteq{\mathbb R}^N$ is convex, we denoteby ${\rm ri }(C)$ 
the relative interior of $C$, i.e. the set of the 
interior points of $C$,
in the topology of ${\it aff}(C)$, once we regard it as a subspace of
${\it aff}(C)$.
We recall that ${\rm ri}(C)\not=\emptyset $ provided $C\not=\emptyset$. When ${\it aff}(C)=\mathbb R^N$ we
write as usual ${\rm ri} (C)={\rm int}(C)$.
Moreover, we also recall that
\begin{equation}
\label{1.1R}
tz+(1-t)z_0\in{\rm ri}(C)\hbox{ whenever 
}z_0\in{\rm ri}(C),\ z\in{\overline C},\hbox{ and 
}t\in[0,1[.
\end{equation}


For every $g:\mathbb R^N\to[0,+\infty]$ we set ${\rm dom} g=\{z\in\mathbb R^N : g(z)<+\infty\}$.

Let $g:\mathbb R^N\to[0,+\infty]$ be convex. Then it is 
well known that ${\rm dom} g$ is convex, that $g$ is 
lower semicontinuous in
${\rm ri}({\rm dom} g)$, and that the restriction of $g$ to 
${\rm ri}({\rm dom} g)$ is continuous. In particular, if 
${\rm int}({\rm dom} g)\not=\emptyset$, then $g$ is
continuous in ${\rm int}({\rm dom} g)$.

For every $g:\mathbb R^N\to[0,+\infty]$ we denote by ${\rm co} 
g$ the convex envelope of $g$, i.e. the function
$${\rm co} g: z\in\mathbb R^N\mapsto\sup\{ 
\phi(z), \phi:\mathbb R^N\to[0,+\infty]: \phi\leq 
g, \phi \hbox{ convex}\}.$$
Clearly, ${\rm co} g$ is convex, and ${\rm co} g(z)\leq 
g(z)$ for every $z\in\mathbb R^N$. Consequently, ${\rm co} g$ 
turns out to be the greatest convex function
on
$\mathbb R^N$ less than or equal to $g$.


For every $g:\mathbb R^N\to[0,+\infty]$ we denote by
$g^{**}$ the convex lower semicontinuous envelope 
of $g$, i.e. the function defined by
\begin{equation}\label{astast}
g^{\ast \ast}(z)=\sup\{\phi(z), \phi :\mathbb R^N\to[0,+\infty], \phi \leq g, \phi \hbox{ convex 
and lower semicontinuous}\}.
\end{equation}
Clearly, $g^{**}$ is convex and lower 
semicontinuous, and $g^{**}(z)\leq g(z)$ for every 
$z\in\mathbb R^N$. Consequently, $g^{**}$ turns out to be
the greatest convex lower semicontinuous function on
$\mathbb R^N$ less than or equal to $g$.

\begin{Proposition}\label{Proposition 1.3R} Let 
$g:\mathbb R^N\to[0,+\infty]$. Then ${\rm ri}({\rm dom}
(g^{**}))={\rm ri}({\rm dom}({\rm co} g))={\rm ri}({\rm co}({\rm dom} g)$, and
$$g^{**}(z)={\rm co} g(z)\hbox{ for every 
}z\in{\rm ri}({\rm co}({\rm dom} g))\cup(\mathbb R^N\setminus\overline{\rm co}({\rm dom} 
g)).$$
\end{Proposition}

We refer to \cite{CDA} for more details.

\noindent
A function $g:\mathbb R^N \to (-\infty, + \infty]$ is said to be {\it level convex} if for every $t \in \mathbb R$, the level set $\{z\in \mathbb R^N: g(z)\leq t\}$ is convex. This property can be equivalently stated as follows: for every $\lambda \in [0,1], z_1,z_2 \in \mathbb R^N$,
$$
g(\lambda z_1 + (1-\lambda)z_2) \leq \max\{g(z_1) , g(z_2)\}. 
$$  
The level convex envelope of a function $g:\mathbb R^N \to ]-\infty, +\infty]$ is defined as
\begin{equation}\label{lcenv}
g^{\rm lc}(z):=\sup\{h(z), h:\mathbb R^N\to ]-\infty;+\infty]: h\leq g, h \hbox{ lower semicontinuous, level convex.} \}
\end{equation}



For every $E\subseteq\mathbb R^N$, 
we denote by $I_E$ the indicator function of $E$ defined as $I_E(x)=0$ if $x \in E$ and $I_E(x)=+\infty$ if $x \in \mathbb R^N \setminus E$.

Let $z \in \mathbb R^N$, we denote by $u_z$ the linear function, defined in $\mathbb R^N$, whose gradient is $z$.

\bigskip



Let $\theta :{\cal A}_0(\mathbb R^N)\to[0,+\infty]$. We say that $\theta$ is increasing if
$$\theta(\O_1)\leq\theta(\O_2)\hbox{ for every }\O_1,\ 
\O_2\in{\cal A}_0(\mathbb R^N)\hbox{ such that }\O_1\subseteq\O_2.$$
We denote by $\theta_-$ the inner regular envelope of $\theta$ defined by
\begin{equation}\label{innreg}
\theta_-: 
\O\in{\cal A}_0(\mathbb R^N)\mapsto\sup\left\{\theta(A) : 
A\in{\cal A}_0(\mathbb R^N), {\overline A}\subseteq\O\right\},
\end{equation}
and say that $\theta$ is inner regular in $\O\in{\cal A}_0(\mathbb R^N)$ if
$$\theta(\O)=\theta_-(\O).$$
It is clear that $\theta_-$ is increasing, and that 
it is inner regular in $\Omega$ for every 
$\Omega\in{\cal A}_0({\mathbb R}^N)$.

Hereafter, if $U$ is a set, $\Phi:
U \times {\cal A}_0({\mathbb R}^N) \to[0,+\infty]$, and $u\in U$, we denote by 
$\Phi_-(u,\cdot)$ the inner regular
envelope of $\Phi(u,\cdot)$, namely the function 
defined, for every $\Omega \in{\cal A}_0({\mathbb R}^N)$, by 
$\Phi_-(u,\Omega)=\Phi(u,\cdot)_-(\Omega)$.

\subsection{$\Gamma$-convergence}\label{Gammaconv}

\begin{Definition} \label{seqcharacsupinf}
Let  $X$ be a metric space, and $F_n : X \to (-\infty,+\infty]$ be a sequence of functions. We denote by 
\begin{equation}\label{igamma}F'(x):=\inf \left\{ \liminf_{n \to \infty} F_n(x_n) : x_n \to x \, \hbox{in } X\right\}
\end{equation}
the {\rm $\Gamma$-lower limit}, or more shortly the
{\rm $\Gamma$-liminf} of the sequence $\{F_n\}$. Similarly, we denote by 
\begin{equation}\label{iigamma}
F''(x):=\inf\left\{\limsup_{n \to \infty} F_n(x_n) : x_n \to x \, \hbox{in } X\right\}
\end{equation}
the {\rm $\Gamma$-upper limit}, or more shortly the
{\rm $\Gamma$-limsup} of the sequence $\{F_n\}$.
When $F'=F''=F$, we say that $F$ is the $\Gamma$-limit of the sequence $\{F_n\}$, and it is characterized by the
following properties:
\begin{itemize}
\item[(i)] for every $x \in X$ and for every sequence $(x_n)$
converging to $x$ in $X$, then
$$F(x)\leq \liminf_{n \to \infty}F_n(x_n); $$
\item[(ii)]for every $x \in X$ there exists a sequence $(\overline x_n)$ (called a {\rm recovering sequence}) converging to $x$ in $X$   such that
$$F(x)= \lim_{n \to \infty} F_n(\overline x_n).  $$
\end{itemize}
\end{Definition}
\begin{Definition}
[$\Gamma$-convergence for a family of functionals]\label{gammafamily} 
Let $X$ be a metric space, and $\{F_\e\} : X \to (-\infty,+\infty]$ be a family of functionals. Definition \ref{seqcharacsupinf} can be extended to $\{F_\e\}$ if \eqref{igamma} and \eqref{iigamma} hold for every sequence $\{\e_n\}$ extracted by the family $\{\e\}$, with $\e_n \to 0$.

\end{Definition}



\bigskip
\subsection{Supremal functionals}\label{supremal}
Let $\Omega \subset \mathbb R^N$ be an open bounded domain. 
A supremal (localized) functional on  $W^{1,\infty}(\Omega)$ is a functional of the form
\begin{equation}\label{Fsup}
F(u,A):=\supess_{x \in A}f(\nabla u(x)),
\end{equation}
where $u \in W^{1,\infty}(\Omega)$ and $A$ is any open subset of $\Omega$. 
The function $f$ which represents the functional is called {\it supremand}. 

As observed in \cite{BJW99}, a necessary and sufficient condition for the lower semicontinuity of a supremal functional of the type \eqref{Fsup} with respect to the weak* topology of $W^{1,\infty}(\Omega)$ is the lower semicontinuity and level convexity of the supremand $f$.

A general relaxation results for supremal functionals (see \cite[Theorem 2.3]{BPZ}), with a Carath\'eodory supremand, with dependence on $x$ and $u$, has been proved in \cite{P} and involves the lower semicontinuous and level convex envelope of the supremand $f$ (c.f. formula \eqref{lcenv}).
 
\section{Integral representation for gradient constrained functionals arising in dimension reduction}\label{unbounded}

This section contains the main steps to the  proof of Theorem \ref{thmDAFL-Z}.

\noindent Here and in the remainder of this section let $\omega \in {\cal A}_0(\mathbb R^2)$ and $\Omega:=\omega \times (-1,1)$. We will write any $x \in \Omega:=\omega \times (-1,1)$
as $(x_\alpha, x_3)$, where $\alpha$ stands for $1$ and $2$.
Moreover $d {\cal L}^3$ will be denoted also as $dx$ or $dx_\alpha d x_3$. 

If $u$ is a Sobolev function, we denote by $\nabla u$ 
the vector of all the partial derivatives of $u$ 
with respect to its variables. If $u$ is defined 
on an
open subset of ${\mathbb R}^3$, we denote by
$\nabla_\alpha u$ the vector of the first
$2$ components of
$\nabla u$ and by
$\nabla_3 u$ the last component.




Let $f:\mathbb R^3\to[0,+\infty]$ be a Borel function. For every $\varepsilon >0$, $\omega\in{\cal A}_0(\mathbb R^2)$  and $u \in W^{1,\infty}(\Omega)$ we define $F_{f_\e}$ as in \eqref{Ffe}.



We set
\begin{equation}\label{1.13'R}
F'(\cdot,\o): u \in W^{1,\infty}(\Omega)\to\Gamma ({L^\infty}) -\liminf_{\e \to 0}F_{f_\e}(u,\o),
\end{equation}

\begin{equation}\label{1.13''R}
F''(\cdot,\o): u \in W^{1,\infty}(\Omega)\to\Gamma (L^\infty)-\limsup_{\e \to 0}F_{f_\e}(u,\o).
\end{equation}

\noindent We recall that, from Definitions \ref{seqcharacsupinf} and \ref{gammafamily} for every $\o\in{\cal A}_0(\mathbb R^2)$, 
$F'(\cdot,\o)$ and $F''(\cdot,\o)$ are 
lower 
semicontinuous with respect to the uniform convergence.
Obviously,
$$
F'(u,\o)\leq 
F''(u,\o)\hbox{ for every }u\in W^{1,\infty}(\Omega).
$$

Finally we set
\begin{equation}\label{1.18R}\displaystyle{f_0: z\in{\mathbb R}^2\mapsto\inf_{\zeta\in\mathbb R}f(z,\zeta).}
\end{equation}
We define 
$f_0^{**}=(f_0)^{**}$. 

It is clear that
\begin{equation}\label{1.30R}
{\rm dom} f_0={\rm Pr}_2{\rm dom} f,
\end{equation}
where ${\rm Pr}_2$ is the projection operator from $\mathbb R^3$ to ${\mathbb R}^2$.

By \eqref{1.30R}, once we recall that the operators 
${\rm co}$ and ${\rm Pr}_2$ commute, we deduce that
\begin{equation}\label{1.31R}
{\rm Pr}_2{\rm co}({\rm dom} f)\subseteq{\rm dom} f_0^{**}\subseteq\overline{{\rm Pr}_2{\rm co}({\rm dom} f)}.
\end{equation}

Nest we state some auxiliary results.
 
\noindent If $g:\mathbb R^2\to[0,+\infty]$ is convex and lower 
semicontinuous, we define the functional $F_g$ as
\begin{equation}\label{1.7R}
F_g:(u,\omega)\in  
W^{1,1}_{\rm loc}(\mathbb R^2)\times {\cal A}_0(\mathbb R^2)\mapsto\int_\omega g(\nabla u)dx.
\end{equation}

The proof of the following result can be found in \cite{CDA}.
\begin{Proposition}\label{Proposition 1.4R} Let 
$g:\mathbb R^2\to[0,+\infty]$ be convex and lower 
semicontinuous, and let $F_g$ be defined in \eqref{1.7R}. Then, for every $\omega\in{\cal A}_0(\mathbb R^2)$, $F_g(\cdot,\omega)$ is $L^1_{\rm loc}(\omega)$-lower
semicontinuous.
\end{Proposition}


The following approximation in energy result for 
$F_g$ holds (cf. \cite[Lemma 7.4.4]{CDA}). Here and 
in what follows for every
$\o\in{\cal A}_0(\mathbb R^2)$ and $\eta>0$, we set
$\o^-_\eta=\{x\in\omega : \dist(x,\partial \o)>\eta\}$.

\begin{Proposition}\label{Proposition 1.6R} Let 
$g:\mathbb R^2\to[0,+\infty]$ be convex and lower 
semicontinuous, and let $F_g$ be defined in \eqref{1.7R}. Then
$$F_g(u_\eta, \o^-_\eta)\le F_g(u,\o)\hbox{ for 
every }\o\in\A_0(\mathbb R^2),\ u\in W^{1,1}_{\rm loc}(\mathbb R^2),\hbox{ 
and }\eta>0,$$
\end{Proposition}
where $\{u_\eta\}$ is the sequence of standard mollifications as defined in \cite[formula (4.1.2)]{CDA}.

\begin{proof}[Proof of Theorem \ref{thmDAFL-Z}.] The Gamma convergence result has been obtained by double inequality. The lower bound is a consequence of Proposition \ref{prop3.0DAFL}. The proof of the upper bound, where the gradient constraints really play a role, relies on techniques of Convex Analysis analogous to those contained in \cite{DA}. The proof, contained in \cite{DAF}  is sketched in subsection \ref{ub}.
\end{proof}

The following result relies on a by now `standard' technique in dimensional reduction (cf. for instance \cite{LDR}).
   
\begin{Proposition}\label{prop3.0DAFL}
Let $f:{\mathbb R}^3\to[0,+\infty]$ be a Borel function as in Theorem \ref{thmDAFL-Z}, and let $F'$ be 
given by \eqref{1.13'R}. Then
$$
2 \int_\omega f_0^{**}(\nabla_\alpha u)dx_\alpha \leq F'(u,\o)\hbox{ 
for every }(u, \omega) \in W^{1,\infty}(\omega)\times {\cal A}_0({\mathbb R}^2).$$
\end{Proposition}
\begin{proof}[Proof.] Let $(u,\o)\in W^{1,\infty}(\omega) \times {\cal A}_0({\mathbb R}^2)$, and let $\{u_\e\}\subseteq W^{1,\infty}(\Omega)$ be a sequence such that $u_\e\to u$ uniformly  and $\displaystyle{\liminf_{\e\to+\infty}F_\e(u_\e,\o)<+\infty}$. Then
$$\liminf_{\e\to 0}F_\e(u_\e, \o)=\liminf_{\e\to 0}\int_{\omega \times (-1,1)} f\left(\nabla_\alpha
u_\e,{\frac{1}{\e}}\nabla_3 u_\e\right)d x_\alpha d x_3 \ge\liminf_{\e\to 0}\int_{\omega\times (-1,1)} f_0^{**}(\nabla_\alpha
u_\e)dx_\alpha.$$

Now, for a.e. $x_3\in (-1,1)$, it turns out that 
$u_\e(\cdot,x_3)\to u$ in $L^1(\omega)$. Consequently, 
Proposition \ref{Proposition 1.4R}, applied to $\displaystyle{\int_\omega f_0^{**}(\nabla u) dx_\alpha}$  implies
$$\liminf_{\e \to 0}\int_{\omega} f_0^{**}(\nabla_\alpha
u_\e(\cdot,x_3))dx_\alpha \geq\int_\omega f_0^{**}(\nabla 
u)dx_\alpha \hbox{ for }\hbox{a.e. }x_3\in (-1,1).$$
Because of this and of Fatou's lemma, we deduce that
$$
\begin{array}{ll}
\displaystyle{\liminf_{\e\to 0}F_\e(u_\e, \o)\geq\int_{-1}^1\left(\liminf_{\e\to 0}\int_\omega 
f_0^{**}(\nabla_\alpha
u_\e(\cdot,x_3))dx_\alpha\right)d x_3\geq}
\\
\\
\displaystyle{\int_{-1}^1\left(\int_\omega f_0^{**}(\nabla 
u)dx_\alpha\right)dx_3=
2\int_\omega f_0^{**}(\nabla u)dx_\alpha.}
\end{array}
$$

This concludes the proof.
\end{proof}

We stress the fact that the lower bound inequality does not involve any technical assumption on $f$ and is not affected from the unboundedness of $f$. 

\bigskip
\bigskip

\subsection{Estimate from above}\label{ub}

We quote the results necessary to the achievement of the $\Gamma$-limsup inequality. For the sake of brevity we do not present all the proofs, which are contained in \cite{DAF}, but we just give the main steps. A more detailed proof is provided for the last two lemmas contained in this subsection, in order to reach the representation in $W^{1,\infty}(\o)$, since the results contained in \cite{DAF} regard lower semicontinuity in $L^1$ and obtain an integral representation on the space of functions of bounded variation.
We also stress the fact that the `upper bound' is achieved first on linear functions, then on piecewise affine functions, smooth ones and finally on $W^{1,\infty}$. To this end we recall that $PA(\mathbb R^2)$ is the set of piecewise affine functions, namely continuous functions, that can be written as
$\displaystyle{u(x)=\sum_{j=1}^m (u_{z_j}(x)+ s_j)\chi_{P_j}(x) \hbox{ for a.e.}x \in \mathbb R^N}$, with $P_j$ polyhedral sets, namely 
intersections of a finite family of closed half-spaces. Clearly, a polyhedral set is closed and convex.

The proof of following lemma relies on a refined version of the so called Zig-Zag lemma, proven in \cite[Lemmas 3.1 and 3.3]{DA}.

\begin{Lemma}\label{Lemma 3.4R} Let $f:{\mathbb R}^3 \to[0,+\infty]$ be a Borel function as in Theorem \ref{thmDAFL-Z}. Then
$$\inf\left\{\limsup_{\e\to+\infty}\int_{\Omega} f\left(\nabla_\alpha u_\e,\frac{1}{\e}\nabla_3 u_\e\right)dx_\alpha dx_3:\right.$$
$$\{u_\e\}\subseteq W^{1,\infty}(\Omega),\ u_\e\to u_z\hbox{ in }L^{\infty}(\Omega))\Big\}\leq 2{\cal L}^2(\omega){\rm co} f_0(z)$$
$$\hbox{for every }(z, \omega) \in \mathbb R^2 \times \in{\cal A}_0({\mathbb R}^2).$$
\end{Lemma}

The lower semicontinuity of $F''$ with respect to the uniform convergence, the properties of $f_0^{\ast \ast}$, \eqref{1.1R} and Proposition \ref{Proposition 1.3R}, the application of Lemma \ref{Lemma 3.4R} to $tz+ (1-t)z_0$ with $z_0 \in {\rm ri}( {\rm dom}f_0^{\ast \ast})$,with  $t \in (0,1)$, allow to prove the following result.
\begin{Proposition}\label{Proposition 3.5R} Let 
$f:\mathbb R^3 \to[0,+\infty]$ be a Borel function as in Theorem \ref{thmDAFL-Z}, and let $F''$ 
be given by \eqref{1.13''R}. Then
$$F''(u_z,\o)\leq 2{\cal L}^2(\omega)f_0^{**}(z)\hbox{ for 
every }(z, \omega)\in \mathbb R^2 \times {\cal A}_0({\mathbb R}^2).$$
\end{Proposition}

The extension of Proposition \ref{Proposition 3.5R} to piecewise 
affine functions relies on a preparatory result analogous to \cite[Lemma 3.6]{CCEDA}, whose proof we omit for the sake of exposition. The statement below is an adaptation of \cite[Lemma 2.1]{DA}
\begin{Lemma}\label{comeLemma3.6}
Let $u=\sum_{j=1}^m 
(u_{z_j}+s_j)\chi_{P_j}$ be in $PA(\mathbb R^2)$, and 
let $\omega$ be a convex open subset of $\mathbb R^2$.
Then there exist $k\in\mathbb N$ and $N_1$, \dots, 
$N_k\subseteq\{j\in\{1,\dots,m\} : 
{\rm int}(P_j)\cap\omega\not=\emptyset\}$ such that
\begin{equation}\label{bla}u(x)=\max_{i\in\{1,\dots,k\}}\min_{j\in 
N_i}(u_{z_j}(x)+s_j)\hbox{ for every }x\in\omega.
\end{equation}
\end{Lemma}

We also underline, that when representing the $\Gamma$-limit on piecewise affine functions, because of the use of representation \eqref{bla}, the domain $\omega$ needs to be assumed convex. In fact, we recall that in \cite{CCEDA} several counterexamples are provided, which show that \eqref{bla} cannot be achieved if the convexity of $\omega$ is dropped.
In fact the result below can be achieved in two steps, first one provides an abstract formula for the integral representation of the Gamma-limsup on  functions which are minimun and/or maximum of others, then Lemma
\ref{comeLemma3.6} and Proposition \ref{Proposition 3.5R} are applied.

\begin{Proposition}\label{Proposition 3.7R}
Let  $f:\mathbb R^3\to[0,+\infty]$ be a Borel function as in Theorem \ref{thmDAFL-Z}, and let $F''$ 
be given by \eqref{1.13''R}. Then
\begin{equation}
\nonumber
F''(u,\o)\leq\int_\omega f_0^{**}(\nabla_\alpha u)dx_\alpha \hbox{ 
for every }\ u\in 
PA(\mathbb R^2), \o\in\A_0(\mathbb R^3)\hbox{ convex}.\end{equation}

\end{Proposition}

Next it is stated the upper bound for the inner regular envelope of $F''$ (cf. \eqref{innreg}) on the class of $C^1(\mathbb R^2)$ functions. The proof can be obtained exploiting the strong approximation of $C^1$ functions by piecewise affine ones and the convexity of $f_0^{\ast \ast}$ and adapting to the dimensional reduction case arguments in the same spirit of those contained in \cite[Lemma 3.8]{DA}. We also would like to underline that the convexity of $\omega$ yields that,
\begin{equation}\nonumber
\sup\left\{F''(u,A) : A\in\A_0(\mathbb R^2),\ A\hbox{ 
convex},\ {\overline 
A}\subseteq \o\right\}=F_-''(u,\omega),
\end{equation}
 and this plays a crucial role to obtain \eqref{C1rep}.

\begin{Lemma}\label{Lemma 3.8R} Let $f:\mathbb R^3\to[0,+\infty]$ 
be a Borel function as in Theorem \ref{thmDAFL-Z}, and let $F''$ be given by \eqref{1.13''R}. Then
\begin{equation}\label{C1rep}
F''_-(u, \o)\leq\int_\o f_0^{**}(\nabla_\alpha 
u)dx_\alpha \hbox{ for every }\ u\in 
C^1(\mathbb R^2), \o\in\A_0(\mathbb R^2)\hbox{ 
convex}.
\end{equation}
\end{Lemma}

We present the completion of the proof of the upper bound inequality. 
We argue as in \cite[Lemma 3.9]{CCEDA}.

\begin{Lemma}\label{Lemma 3.9CCEDA}
 Let $f:\mathbb R^3\to[0,+\infty]$ 
be a Borel function as in Theorem \ref{thmDAFL-Z}, and let $F''$ be given by \eqref{1.13''R}. Then
\begin{equation}\label{3.53CCEDA}
F''_-(u,\o)\leq\int_\o f_0^{**}(\nabla_\alpha
u)dx_\alpha \hbox{ for every }\ u\in 
W^{1,\infty}_{\rm loc}(\mathbb R^2), \o\in\A_0(\mathbb R^2)\hbox{ 
convex}.
\end{equation}
\end{Lemma}

\begin{proof}[Proof]
Let $\omega$ and $u$ as in \eqref{3.53CCEDA}.
Denoting by $u_\eta(x_\alpha)$ the regularization of $u$ at $x_\alpha$ defined in \cite[formula (4.2.1)]{CDA}, and recalling that 
for every $\eta>0$, $u_\eta\in
C^\infty(\mathbb R^2)$, and that $u_\eta\to u$ uniformly on the compact sets of $\mathbb R^2$ as $\eta\to 0$,
by Proposition \ref{Proposition 1.6R} applied to $F_{f_0^{**}}$, we have
 
\begin{equation}\label{3.54CCEDA}
\int_{\omega^-_\eta}f_0^{**}(\nabla_\alpha u_\eta) dx_\alpha \leq \int_\o f_0^{\ast \ast}(\nabla_\alpha u) dx _\alpha \hbox{ for 
every } u\in W^{1,\infty}_{\rm loc}(\mathbb R^2),\hbox{ 
and }\eta>0.
\end{equation}

Let now $A \subset \subset \omega$ be a convex open set, and take $\eta>0$ sufficiently small in order to guarantee that $A \subset \subset \omega^-_\eta$. The convexity of $A$, Lemma \ref{Lemma 3.8R} and  \eqref{3.54CCEDA} provide
\begin{equation}\label{3.55CCEDA}
\begin{array}{ll}
\displaystyle{(F'')_-(u_\eta, A)\leq \int_A f_0^{\ast \ast}(\nabla_\alpha u_\eta)dx_\alpha \leq \int_{\omega^-_\eta}f_0^{\ast \ast}(\nabla_\alpha u_\eta)dx_\alpha \leq}\\
\\
\displaystyle{\int_\omega f_0^{\ast \ast}(\nabla u )dx_\alpha \hbox{ for every }\eta>0 \hbox{ sufficiently small.}}
\end{array}
\end{equation}

Passing to the limit in \eqref{3.55CCEDA} first as $\eta \to 0$, then letting $A$ invade $\o$, and exploiting the lower semicontinuity of $(F'')_-$ (see \cite[Remark 15.13]{D}) and the fact that $(F'')_-$ is an increasing set function, the thesis follows.  
\end{proof}

The proof of the upper bound inequality is concluded by the following lemma, in the spirit of \cite[Theorem 3.10]{CCEDA}

\begin{Lemma}\label{Theorem 3.10R} Let 
$f:\mathbb R^3\to[0,+\infty]$ be Borel,  
and $F''$ be given by \eqref{1.13''R}. Then
$$\displaystyle{F''(u,\o)\leq \int_\o f_0^{**}(\nabla_\alpha 
u)dx_\alpha \hbox{ for every }u\in W^{1,\infty}(\o), \o\in\A_0(\mathbb R^2)\hbox{ convex}.}$$
\end{Lemma}
\begin{proof}[Proof] 

Since $\omega$ is convex there is no loss of generality in assuming $u \in W^{1,\infty}(\mathbb R^2)$. 
Analogously we can assume that ${\rm dom}f_0^{\ast \ast}$ is not empty.
If ${\rm dom}f_0^{\ast \ast}$ contains just one point the thesis follows by Proposition \ref{Proposition 3.5R}. 
Thus we can suppose also that $(0,0) \in {\rm dom}f_0^{\ast \ast}$, consequently the convexity of $f_0^{\ast \ast}$ entails that $(0,0) \in {\rm ri}({\rm dom}f_0^{\ast \ast})$.

\noindent If this was not the case we could assume that $u_0 \in \mathbb R^2$ is the element in ${\rm ri}({\rm dom}f_0^{\ast \ast})$, and the subsequent analysis could be repeated for the translated function $g_0^{\ast \ast}(z):= f_0^{\ast \ast}(z+ z_0)$ and the translated functional $F_{u_0}''(u, \omega):= F''(u_{z_0}+ z, \omega)$.

Let $x_0 \in \omega$ and $t>1$ and let $u_t$ be the function defined 
by $u_t: x_\alpha\in \mathbb R^2 \to u\left(x_0+\frac{x_\alpha-x_0}{t}\right)$, then $u_t \in W^{1,\infty}(\mathbb R^2)$ and $u_t \to u$ in $L^\infty(\omega)$. The convexity of $\omega$ and Lemma \ref{Lemma 3.9CCEDA} entail that
$$
\displaystyle{F''(u_t, \omega) \leq (F'')_-(u_t, x_0+ t (\omega -x_0)) \leq \int_{x_0 + t(\omega -x_0)} f_0^{\ast \ast}(\nabla_\alpha u_t) dx_\alpha.}
$$
By the change of variables $y_\alpha = x_0 + \frac{x_\alpha -x_0}{t}$ in the right hand side of the inequality above and by the convexity of $f_0^{\ast \ast}$ we have
$$
\begin{array}{ll}
\displaystyle{F''(u_t, \o) \leq t^2 \int_\omega f_0^{\ast \ast}\left(\frac{\nabla_\alpha u}{t}\right)d y_\alpha \leq t\int_\omega f_0^{\ast \ast}(\nabla_\alpha u) d y_\alpha +t^2\left(1-\frac{1}{t}\right)|\omega|f_0^{\ast \ast}(0).} 
\end{array}
$$
The thesis follows as $t \to 1$ by the lower semicontinuity of $F''$.
\end{proof}

\subsection{Theorem \ref{myproof}}\label{thm1.1}

This section is devoted to the proof of Theorem \ref{myproof}. The strategy of the proof is inspired by \cite{BGP}, besides the authors therein consider cell-formulas for the densities. On the other hand, in our context we can avoid the introduction of cell-formulas, and we can deal just with simpler formulas which explicitly provide a representation and exhibit the link with the original densities. 
We start by recalling some properties of the functions involved in formula \eqref{Gammaesssup} (for more details we refer to \cite[Section 6]{BPZ}).

\begin{itemize}
\item[i)] If $W:\mathbb R^3 \to [0, +\infty)$ is continuous and level convex, then $W_0$ is level convex as well.
\item[ii)] If $W: \mathbb R^3 \to [0, +\infty)$ is continuous, then 
$(W_0)^{\rm lc}(z)= (W^{\rm lc})_0(z)$ for every $z \in \mathbb R^2$.
\item[iii)] $(W_0)^{\rm lc}$ is lower semicontinuos.  
\end{itemize}
\begin{proof}[Proof of Theorem \ref{myproof}]
The proof of the lower bound inequality, i.e. the fact that for any $u\in C({\overline \Omega})$ it results
$$
\displaystyle{F_0(u) \leq \liminf_{\e \to 0}F_\e(u_\e)}
$$
whenever $\{u_\e\}\subset C({\overline \Omega})$ is uniformly converging to $u$ in ${\overline \Omega}$, is exactly as in the proof of \cite[Theorem 6.1]{BPZ}.

For what  concerns the upper bound we observe that \cite[Theorem 2.6]{P} guarantees that there is no loss of generality in assuming $W$ already level convex and the right hand side of \eqref{Gammaesssup} finite.

The result will be achieved through a repeated application of Theorem \ref {thmDAFL-Z}.

Let $u \in W^{1,\infty}(\omega)$ and let $\displaystyle{M:= {\rm ess}\sup_\omega (W_0)^{\rm lc}(\nabla_\alpha u)}$ and consider  the functionals 
$$u\in W^{1,\infty}( \Omega) \to {\cal I_C}(u)=
\displaystyle{\int_\Omega I_C\left(\nabla_\alpha u,\frac{1}{\varepsilon}\nabla_3 u\right)dx}
$$ where $I_{C}$ is the indicator function of $C= \left\{z \in \mathbb R^3: W(z)\leq M\right\}.$ 
Since $W$ is continuous and level convex it  is easily observed that $C$  is closed and convex. \eqref{Ccoercivity} entails that $C$ is also bounded.
Consequently $I_C$ is lower semicontinuous. These facts allow us to apply Theorem \ref{thmDAFL-Z} which provides \begin{equation}\label{Gammalimit}
\displaystyle{\Gamma(L^\infty)-\lim_{\e \to 0} {\cal I_C}(u_\e)=
2\int_\omega (({I_C})_0)^{\ast \ast}(\nabla_\alpha u)dx_\alpha }
\end{equation}
if $u \in W^{1,\infty}(\omega)$.

It is crucial to observe that if $C_0=\{z \in \mathbb R^2: (W_0)^{\rm lc}(z)\leq M\}$, then $I_{C_0}(z)=  (({I_C})_0)^{\ast \ast}(z)$. 

\noindent In fact, the assumptions on $W$ and ii) imply that $(W_0)^{\rm lc}=W_0$, moreover \eqref{1.30R} entails that 
$$C_0 = {\rm Pr}_2 C.$$
The last equality, \eqref{1.31R} and the fact that $C$ is closed and convex entail that 
$$
{\rm dom}(((I_{C})_0)^{\ast \ast})= {\rm dom}(I_{C_0}).
$$
\eqref{Gammalimit} guarantees the existence of a sequence $\{u_\varepsilon\}_\varepsilon$ such that $u_\varepsilon \to u$ uniformly and 

$$
\displaystyle{\limsup_{\e \to 0}\int_\Omega I_C\left(\nabla_\alpha u_\e,\frac{1}{\e}\nabla_3 u_\e \right)dx \leq 2 \int_\omega I_{C_0}(\nabla_\alpha u)dx_\alpha},
$$
thus we can say that there exists $\e_0>0$ such that for every $\e < \e_0$
$$
\displaystyle{I_C\left(\nabla_\alpha u_\e, \frac{1}{\e}\nabla_3 u_\e\right)=0 \hbox{ for a.e. }x \in \Omega},
$$
i.e. for every $\e <\e_0$ (cf. \eqref{remunb})
$$
\displaystyle{\supess_\Omega W\left(\nabla_\alpha u_\e, \frac{1}{\e}\nabla_3 u_\e\right) \leq M \hbox{ for a.e. } x \in \Omega,}
$$
equivalently
$$\limsup_{\varepsilon \to 0} \supess_{\Omega}W\left(\nabla_\alpha u_\varepsilon, \frac{1}{\varepsilon}\nabla_3 u_\varepsilon\right)\leq  \supess_\omega (W_0)^{\rm lc}(\nabla_\alpha u).
$$
This concludes the proof.

\end{proof}
 \textbf{Acknowledgments.}

The author wish to thank Irene Fonseca. The author is also deeply indebted with Riccardo De Arcangelis, to whose memory this note is dedicated.

\end{document}